\documentclass{amsart}
\usepackage{amsmath, amsthm, amssymb, graphics}

\allowdisplaybreaks

\numberwithin{equation}{section}

\newcommand{\Z}{\mathbb{Z}}
\newcommand{\N}{\mathbb{N}}
\newcommand{\R}{\mathbb{R}}
\newcommand{\F}{\mathcal{F}}
\newcommand{\Sh}{\mathcal{S}}

\newcommand{\supp}{\mathop{\mathrm{supp}}}
\newcommand{\pa}{\partial}
\newcommand{\la}{\langle}
\newcommand{\ra}{\rangle}
\newcommand{\K}{\boldsymbol{k}}
\newcommand{\Nu}{\boldsymbol{\nu}}

\theoremstyle{plain}
\newtheorem{thm}{Theorem}[section]
\newtheorem{prop}[thm]{Proposition}
\newtheorem{lem}[thm]{Lemma}

\theoremstyle{definition}

\begin{document}
\title[Bilinear pseudo-differential operators of $S_{0,0}$-type]
{On the ranges of bilinear pseudo-differential operators
of $S_{0,0}$-type on $L^2 \times L^2$}

\author[N. Hamada]{Naoki Hamada}
\author[N. Shida]{Naoto Shida}
\author[N. Tomita]{Naohito Tomita}


\address{Department of Mathematics, 
Graduate School of Science, Osaka University, 
Toyonaka, Osaka 560-0043, Japan}

\email[N. Hamada]{u171235f@ecs.osaka-u.ac.jp}
\email[N. Shida]{u331453f@ecs.osaka-u.ac.jp}
\email[N. Tomita]{tomita@math.sci.osaka-u.ac.jp}



\begin{abstract}
In this paper,
the ranges of bilinear pseudo-differential operators
of $S_{0,0}$-type on $L^2 \times L^2$
are determined in the framework of Besov spaces.
Our result improves
the $L^2 \times L^2 \to L^1$ boundedness
of those operators with symbols
in the bilinear H\"ormander class
$BS^m_{0,0}$.
\end{abstract}

\maketitle

\section{Introduction}
In this paper,
we use the following two symbol classes of $S_{0,0}$-type.
One is the bilinear H\"ormander class $BS^{m}_{0,0}$, $m \in \R$,
consisting of all $\sigma(x,\xi_1,\xi_2) \in C^{\infty}((\R^n)^3)$ such that
\[
|\partial^{\alpha}_{x}\partial^{\beta_1}_{\xi_1}\partial^{\beta_2}_{\xi_2}
\sigma(x,\xi_1,\xi_2)|
\le C_{\alpha,\beta_1,\beta_2}(1+|\xi_1|+|\xi_2|)^m
\]
for all multi-indices $\alpha,\beta_1,\beta_2 \in \N_0^n=\{0,1,2,\dots \}^n$.
The other is $BS^{(m_1,m_2)}_{0,0}$, $m_1,m_2 \in \R$,
consisting of all $\sigma(x,\xi_1,\xi_2) \in C^{\infty}((\R^n)^3)$ such that
\[
|\partial^{\alpha}_{x}\partial^{\beta_1}_{\xi_1}\partial^{\beta_2}_{\xi_2}
\sigma(x,\xi_1,\xi_2)|
\le C_{\alpha,\beta_1,\beta_2}(1+|\xi_1|)^{m_1}(1+|\xi_2|)^{m_2} .
\]
For a symbol $\sigma$,
the bilinear pseudo-differential operators $T_{\sigma}$ is defined by
\[
T_{\sigma}(f_1,f_2)(x)
=\frac{1}{(2\pi)^{2n}}
\int_{(\R^n)^2}e^{ix\cdot(\xi_1+\xi_2)}
\sigma(x,\xi_1,\xi_2)
\widehat{f_1}(\xi_1)\widehat{f_2}(\xi_2)\, d\xi_1d\xi_2
\]
for $f_1,f_2 \in \Sh(\R^n)$.

In the linear case,
the celebrated Calder\'on-Vaillancourt theorem \cite{CV} states that
if a symbol $\sigma(x,\xi) \in C^{\infty}((\R^n)^2)$ satisfies
\begin{equation}\label{linear}
|\partial^{\alpha}_{x}\partial^{\beta}_{\xi}\sigma(x,\xi)|
\le C_{\alpha,\beta},
\end{equation}
then the corresponding pseudo-differential operator
is bounded on $L^2$.
As a bilinear counterpart of this theorem,
we naturally expect that
the condition $\sigma \in BS^0_{0,0}$
ensures the $L^2 \times L^2 \to L^1$ boundedness
of $T_{\sigma}$.
However, B\'enyi-Torres \cite{BT} pointed out
that this boundedness does not hold in general.
Michalowski-Rule-Staubach \cite{MRS}
treated the subcritical case $m<-n/2$,
and then Miyachi-Tomita \cite{MT-IUMJ} showed that
all bilinear pseudo-differential operators
with symbols in $BS^m_{0,0}$ are bounded from
$L^2 \times L^2$ to $L^1$ if and only if $m \le -n/2$.
It should be mentioned that
the stronger $L^2 \times L^2 \to h^1$ boundedness
was proved in \cite{MT-IUMJ}, where $h^1$ is the local Hardy space.
See also Miyachi-Tomita \cite{MT}
for the classes of $S_{\rho,\rho}$-type,
$0 \le \rho<1$.

Kato-Miyachi-Tomita recently proved that
if $m_1,m_2<0$,
$m_1+m_2=-n/2$ and $\sigma \in BS^{(m_1,m_2)}_{0,0}$,
then $T_{\sigma}$ is bounded from $L^2 \times L^2$ to
the amalgam space $(L^2,\ell^1)$
(\cite[Theorem 1.3, Example 1.4]{KMT}).
Here, we remark that
\begin{equation}\label{two-classes}
BS^{-n/2}_{0,0} \subset BS^{(m_1,m_2)}_{0,0},
\quad m_1,m_2 \le 0, \ m_1+m_2=-n/2.
\end{equation}
Hence, for $1 \le p \le 2$,
the $L^2 \times L^2 \to L^p$ boundedness
of $T_{\sigma}$ with $\sigma \in BS^{-n/2}_{0,0}$ follows
from the embedding $(L^2,\ell^1) \hookrightarrow L^p$
(see Section \ref{section2}).
In the linear case,
since the constant function satisfies \eqref{linear},
the $L^2 \to L^p$ boundedness of pseudo-differential operators
with symbols satisfying \eqref{linear} holds
if and only if $p=2$.
On the other hand,
in the bilinear case,
we have a choice to choose target spaces.
Thus, the purpose of this paper
is to determine the ranges of bilinear pseudo-differential operators
with symbols in the critical class $BS^{-n/2}_{0,0}$ on $L^2 \times L^2$
in the framework of Besov spaces $B_{p,q}^s$.

The main result of this paper is the following.

\begin{thm} \label{main1}
Let $m_1, m_2 < 0$ and $m_1+m_2=-n/2$.
If $1 \le p \le 2$, then all bilinear pseudo-differential operators
with symbols in $BS^{(m_1,m_2)}_{0,0}$ are bounded from $L^2(\R^n) \times L^2(\R^n)$
to $B_{p,1}^0(\R^n)$.
\end{thm}

Since $B_{1, 1}^0 \hookrightarrow L^1$
(in fact, $B_{1,1}^0 \hookrightarrow h^1 \hookrightarrow L^1$),
Theorem \ref{main1} is an improvement of the $L^2 \times L^2 \to L^1$ boundedness.

By \eqref{two-classes},
the following means the optimality of Theorem \ref{main1}.

\begin{thm} \label{main2}
Let $0 < p, q \leq \infty$.
Then all bilinear pseudo-differential operators with symbols in $BS^{-n/2}_{0,0}$
are bounded from $L^2(\R^n) \times L^2(\R^n)$ to $B^0_{p, q}(\R^n)$
if and only if $1 \leq p \leq 2$ and $1 \leq q \leq \infty$.
\end{thm}

By the ``only if" part of Theorem \ref{main2}
and the embedding $B_{p,q_1}^s \hookrightarrow B_{p,q_2}^0$
for $0<p,q_1,q_2 \le \infty$ and $s>0$,
if the $L^2 \times L^2 \to B_{p,q}^s$ boundedness
of all $T_{\sigma}$ with $\sigma \in BS^{-n/2}_{0,0}$ holds,
then $s$ must satisfy $s \le 0$.
This is the reason why
we consider the case $s=0$ in Theorems \ref{main1} and \ref{main2}.

The following says that
Theorem \ref{main1} cannot be compared with the result of \cite{KMT}.

\begin{prop}\label{main3}
There is no embedding relation between
the amalgam space $(L^2, \ell^1)(\R^n)$
and the Besov spaces $B_{p,1}^0(\R^n)$, $1 \le p \le 2$.
\end{prop}

Related results to bilinear Fourier multiplier operators of $S_{0,0}$-type
on $L^2 \times L^2$ can be found in
Grafakos-He-Slav\'ikov\'a \cite{GHS}
and Slav\'ikov\'a \cite{Slavikova}.

The contents of this paper are as follows.
In Section \ref{section2},
we give preliminary facts.
In Section \ref{section3},
we give basic estimates used in the proof of Theorem \ref{main1}.
In Sections \ref{section4}, \ref{section5} and \ref{section6},
we prove Theorems \ref{main1}, \ref{main2}
and Proposition \ref{main3}, respectively.

\section{Preliminaries}\label{section2}
For two nonnegative quantities $A$ and $B$,
the notation $A \lesssim B$ means that
$A \le CB$ for some unspecified constant $C>0$,
and $A \approx B$ means that
$A \lesssim B$ and $B \lesssim A$.
For $1 \le p \le \infty$,
$p'$ is the conjugate exponent of $p$,
that is, $1/p+1/p'=1$.
The usual inner product of $f, g \in L^2(\R^n)$ is denoted by 
$\la f, g \ra$ .

Let $\Sh(\R^n)$ and $\Sh'(\R^n)$ be the Schwartz space of
rapidly decreasing smooth functions on $\R^n$ and its dual,
the space of tempered distributions, respectively.
We define the Fourier transform $\F f$
and the inverse Fourier transform $\F^{-1}f$
of $f \in \Sh(\R^n)$ by
\[
\F f(\xi)
=\widehat{f}(\xi)
=\int_{\R^n}e^{-ix\cdot\xi} f(x)\, dx
\quad \text{and} \quad
\F^{-1}f(x)
=\frac{1}{(2\pi)^n}
\int_{\R^n}e^{ix\cdot \xi} f(\xi)\, d\xi.
\]
For $m \in L^{\infty}(\R^n)$,
the Fourier multiplier operator $m(D)$ is defined by
$m(D)f=\F^{-1}[m\widehat{f}]$ for $f \in \Sh(\R^n)$.

Let $\{\psi_\ell \}_{\ell \geq 0} $ be a sequence of Schwartz functions on $\R^n$ satisfying
\begin{align} \label{Besov-parti}
\begin{split}
&\supp \psi_0 \subset \{\xi \in \R^n \,:\, |\xi| \leq 2 \},\\
&\supp \psi_\ell \subset \{ \xi \in \R^n \,:\, 2^{\ell -1 }\leq |\xi| \leq 2^{\ell + 1}\},
\quad \ell\geq 1,
\\
&|\partial^{\alpha}\psi_\ell (\xi)| \leq C_\alpha 2^{-\ell|\alpha|}, \quad \ell \geq 0,
\ \alpha \in \N^n_0, \ \xi \in \R^n,
\\
&\sum_{\ell=0}^\infty \psi_\ell(\xi) = 1, \quad \xi \in \R^n.
\end{split}
\end{align}
For $0 < p, q \leq \infty$ and $s \in \R$,
the Besov space $B^s_{p, q}(\R^n)$ consists of all $f \in \Sh'(\R^n)$
such that
\[
\|f\|_{B^s_{p, q}}=
\bigg(  \sum_{\ell=0}^\infty 2^{\ell sq} \|\psi_\ell(D)f\|^q_{L^p} \bigg)^{1/q}<\infty
\]
with usual modification when $q=\infty$.
It is well known that the definition of Besov spaces
$B_{p,q}^s$ is independent of the choice
of $\{\psi_\ell\}_{\ell \geq 0}$ satisfying \eqref{Besov-parti}.
For $0 < p, q_1,q_2 \le \infty$ and $s_1,s_2 \in \R$,
the embedding $B_{p,q_1}^{s_1} \hookrightarrow B_{p,q_2}^{s_2}$ holds
if $s_1=s_2$ and $q_1 \le q_2$, or if $s_1>s_2$.
The dual space of $B^s_{p, q}$
coincides with $B^{-s}_{p^\prime, q^\prime}$,
where $1 \leq p, q < \infty$ and $s \in \R$.
See \cite{Triebel} for more details on Besov spaces.

For $1 \leq p,q \leq \infty$, 
the amalgam space 
$(L^p,\ell^q)(\R^n)$ consists of all measurable functions $f$ on 
$\R^n$ such that 
\begin{equation*}
\| f \|_{ (L^p,\ell^q)} 
=\bigg\{ \sum_{\nu \in \Z^n}
\bigg( \int_{\nu+[-1/2,1/2]^n} \big| f(x) \big|^p \, dx 
\bigg)^{q/p} \bigg\}^{1/q} 
< \infty  
\end{equation*}
with usual modification when $p$ or $q$ is infinity.  
Obviously, $(L^p,\ell^p) = L^p$,
and $(L^{p_1}, \ell^{q_1})  
\hookrightarrow (L^{p_2},\ell^{q_2})$
for $p_1 \geq p_2$ and $q_1 \leq q_2$.
In particular, 
$(L^2,\ell^1) \hookrightarrow L^r$ 
for $1 \leq r \leq 2$,
and 
the stronger embedding 
$(L^2,\ell^1) \hookrightarrow h^1$ holds
in the case $r=1$,
where $h^1$ is the local Hardy space
(see \cite[Section 2.3]{KMT-2}).

We end this section by quoting the following, which is called Schur's lemma
(see, e.g., \cite[Appendix A]{Grafakos-Modern}).

\begin{lem}\label{Schur's lemma}
Let $\{A_{k_1, k_2}\}_{k_1, k_2 \ge 0}$ be a sequence of nonnegative numbers satisfying
\begin{align*}
\sup_{k_1 \ge 0} \sum_{k_2 \ge 0} A_{k_1, k_2} < \infty
\quad \text{and} \quad
\sup_{k_2 \ge 0} \sum_{k_1 \ge 0} A_{k_1, k_2} < \infty.
\end{align*}
Then
\begin{align*}
\sum_{k_1, k_2 \ge 0} A_{k_1, k_2} b_{k_1} c_{k_2} 
\lesssim \Big(\sum_{k_1 \ge 0} b_{k_1}^2 \Big)^{1/2}\Big(\sum_{k_2 \ge 0} c_{k_2}^2 \Big)^{1/2}
\end{align*}
for all nonnegative sequences $\{b_{k_1}\}_{k_1 \ge 0}$ and $\{c_{k_2}\}_{k_2 \ge 0}$.
\end{lem}

\section{Basic estimates}\label{section3}
For $R > 0$ and $L > n$, we set 
\begin{align*}
  S_R(f)(x) = R^n\int_{\R^n} \frac{|f(y)|}{(1+R|x-y|)^L}\, dy,
\end{align*}
and simply write $S_1(f)(x) = S(f)(x)$.

 The following lemma can be found in the proof of \cite[Lemma 4.2]{Kato}, but we will give the proof for the reader's convenience.

\begin{lem} \label{lem-dyadic-est}
 Let $\varphi \in \Sh(\R^n)$.  Then, we have
 \begin{align*}
  \bigg(\sum_{\nu \in \Z^n} |\varphi(R^{-1}(D -\nu))f(x)|^2 \bigg)^{1/2} 
  \lesssim R^{n/2}S_{R}(|f|^2)(x)^{1/2}
 \end{align*}
 for all $R \ge 1$, where $\varphi(R^{-1}(D -\nu))f = \F^{-1}[\varphi(R^{-1}(\cdot -\nu)) \widehat{f}]$.
\end{lem}

\begin{proof}
Since $\R^n = \cup_{\mu \in \Z^n} (2\pi\mu + [-\pi, \pi]^n)$,
we have
 \begin{align*}
 &\varphi(R^{-1}(D-\nu))f(x)
 = R^n \int_{\R^n} e^{i\nu\cdot(x-y)} 
   \Phi(R(x-y)) f(y) \,dy \\
  &= R^n e^{i\nu\cdot x}\sum_{\mu \in \Z^n} \int_{2\pi\mu + [-\pi, \pi]^n} e^{-i\nu\cdot y} 
   \Phi(R(x-y)) f(y) \,dy\\
  &= R^n e^{i\nu\cdot x}  \int_{[-\pi, \pi]^n} e^{-i\nu\cdot y} 
   \Big(
   \sum_{\mu \in \Z^n} \Phi(R(x-y-2\pi\mu)) f(y+2\pi\mu)
   \Big)\,
    dy,
 \end{align*}
where $\Phi = \F^{-1}\varphi$.
Here, $\sum_{\mu \in \Z^n} \Phi(R(x-y-2\pi\mu)) f(y+2\pi\mu)$
is a $(2\pi \Z)^n$-periodic function of the $y$-variable. Hence, it follows
from Parseval's identity that
 \begin{align*}
  &\Big( \sum_{\nu \in \Z^n} |\varphi(R^{-1}(D-\nu))f(x)|^2 \Big)^{1/2} \\
  &=(2\pi)^{n/2} R^n 
  \Big(
  \int_{[-\pi, \pi]^n} \Big| \sum_{\mu \in \Z^n}\Phi(R(x-y-2\pi\mu)) f(y+2\pi\mu) \Big|^2 dy
  \Big)^{1/2}.
 \end{align*}
 Since
 \begin{align*}
  \sup_{z \in \R^n}\sum_{\mu \in \Z^n} |\Phi(R(z-2\pi\mu))|
   &\lesssim 
   \sup_{z \in \R^n}\sum_{\mu \in \Z^n} (1+R|z-2\pi\mu|)^{-(n+1)}\\
   &\leq 
   \sup_{z \in \R^n}\sum_{\mu \in \Z^n} (1+|z-2\pi\mu|)^{-(n+1)}
   \lesssim 1,
 \end{align*}
 where we used the assumption $R \ge 1$ in the second inequality, we have by Schwarz's inequality
 \begin{align*}
  &\Big(
  \int_{[-\pi, \pi]^n} \Big| \sum_{\mu \in \Z^n}\Phi(R(x-y-2\pi\mu)) f(y+2\pi\mu) \Big|^2 dy
  \Big)^{1/2}\\
  &\lesssim 
  \Big(
  \int_{[-\pi, \pi]^n}  \sum_{\mu \in \Z^n} |\Phi(R(x-y-2\pi\mu))| |f(y+2\pi\mu)|^2 dy
  \Big)^{1/2}\\
  &=\Big(
      \int_{\R^n} |\Phi(R(x-y))| |f(y)|^2 dy
      \Big)^{1/2}.
 \end{align*}
 Since $\Phi$ is a rapidly decreasing function, this gives the desired estimate.
\end{proof}

In the rest of this section, we assume that $m_1, m_2 \in \R$ and
$\sigma \in BS^{(m_1, m_2)}_{0, 0}$, and  shall consider the decomposition of $\sigma$. 
Let $\{\psi_j\}_{j \geq 0}$ be as in (\ref{Besov-parti}), and let $\varphi \in \Sh(\R^n)$
be such that
\begin{align*}
 \supp \varphi \subset [-1, 1]^n
 \quad \text{and} \quad
 \sum_{\nu \in \Z^n} \varphi(\xi- \nu) = 1, \quad \xi \in \R^n.
\end{align*} 
Using these functions, we decompose $\sigma$ as 
\begin{align} \label{decomp-symb}
\sigma(x, \xi_1, \xi_2) 
= \sum_{j \in \N_0} \sum_{\substack{\K \in (\N_0)^2 \\ \K= (k_1, k_2)}} \sigma_{j, \K} (x, \xi_1, \xi_2)
= \sum_{j \in \N_0} \sum_{\substack{\K \in (\N_0)^2 \\ \K= (k_1, k_2)}}  \sum_{\substack{\Nu \in (\Z^n)^2 \\ \Nu= (\nu_1, \nu_2)}} \sigma_{j, \K, \Nu} (x, \xi_1, \xi_2)
\end{align}
with
\begin{align*}
\sigma_{j, \K} (x, \xi_1, \xi_2) 
&=  [\psi_j(D_x)\sigma](x, \xi_1, \xi_2)\psi_{k_1}(\xi_1)\psi_{k_2}(\xi_2)\\
&= \Big( 
    \int_{\R^n} \F^{-1}\psi_j(y) \sigma(x-y, \xi_1, \xi_2)\, dy 
    \Big) \psi_{k_1}(\xi_1)\psi_{k_2}(\xi_2)
\end{align*}
and
\begin{align*}
\sigma_{j, \K, \Nu} (x, \xi_1, \xi_2) =
\sigma_{j, \K} (x, \xi_1, \xi_2)  \varphi(\xi_1-\nu_1) \varphi(\xi_2-\nu_2).
\end{align*}

\begin{lem} \label{lem-symb-est}
$(1)$
For each $\beta_1, \beta_2 \in (\N_0)^n$ and $N \in \N_0$, we have
\begin{align*}
|\pa^{\beta_1}_{\xi_1} \pa^{\beta_2}_{\xi_2} \sigma_{j, \K, \Nu}(x, \xi_1, \xi_2)| \lesssim 2^{k_1m_1 +k_2m_2 -jN}
\end{align*}
for  all $j \in \N_0, \,  \K = (k_1, k_2) \in (\N_0)^2$ and $\Nu \in (\Z^n)^2$.\\
$(2)$ For each $N \in \N_0$, we have
\begin{align*}
|T_{\sigma_{j, \K,\Nu}}(f_1, f_2)(x) | \lesssim 2^{k_1m_1 +k_2m_2 -jN}  S(f_1)(x) S(f_2)(x)
\end{align*}
for all $j \in \N_0, \, \K = (k_1, k_2) \in (\N_0)^2$ and $\Nu \in (\Z^n)^2$.
\end{lem}

\begin{proof}
We first prove the assertion $(1)$.
Let $j \geq 1$. By the moment condition of $\F^{-1}\psi_j$
and Taylor's formula, we have
\begin{align*}
 &\sigma_{j, \K, \Nu}(x, \xi_1, \xi_2) \\
 &= \int_{\R^n} \F^{-1}\psi_j(y) 
   \Big( 
   \sigma_{\K, \Nu}(x-y, \xi_1, \xi_2) 
   - \sum_{|\alpha| < N} \frac{(-y)^{\alpha}}{\alpha!} 
   [\partial_x^\alpha     
   \sigma_{\K, \Nu}](x, \xi_1, \xi_2) 
   \Big) \,dy\\
 &= \int_{\R^n} \F^{-1}\psi_j(y)
   \Big(
   N \sum_{|\alpha| = N}  \frac{(-y)^{\alpha}}{\alpha!} 
   \int_0^1 (1-t)^{N-1} [\partial_x^\alpha     
   \sigma_{\K, \Nu}](x-ty, \xi_1, \xi_2)\, dt
   \Big) \,dy ,
\end{align*}
where
\[
\sigma_{\K, \Nu}(x, \xi_1, \xi_2)
=\sigma(x,\xi_1,\xi_2)
\psi_{k_1}(\xi_1)\psi_{k_2}(\xi_2)
\varphi(\xi_1-\nu_1) \varphi(\xi_2-\nu_2) .
\]
Since $1 + |\xi_i| \approx 2^{k_i}$ for $\xi_i \in \supp \psi_{k_i}$, 
we see that
\begin{align*}
 |\partial_x^\alpha \partial_{\xi_1}^{\beta_1} \partial_{\xi_2}^{\beta_2}      
   \sigma_{\K, \Nu}(x, \xi_1, \xi_2)| \lesssim 2^{k_1m_1+k_2m_2}.
\end{align*}
Hence,
\begin{align*}
 &|\partial_{\xi_1}^{\beta_1} \partial_{\xi_2}^{\beta_2}      
 \sigma_{j, \K, \Nu}(x, \xi_1, \xi_2)|\\
  &\lesssim \int_{\R^n} |\F^{-1}\psi_j(y)|
   \Big(
   \sum_{|\alpha| = N}  |y^{\alpha}|
   \int_0^1 |[\partial_x^\alpha \partial_{\xi_1}^{\beta_1} \partial_{\xi_2}^{\beta_2}  
   \sigma_{\K, \Nu}](x-ty, \xi_1, \xi_2)|\,dt
   \Big) \,dy\\
 &\lesssim 2^{k_1m_1+k_2m_2} \int_{\R^n} |\F^{-1}\psi_j(y)| |y|^N \, dy
 \lesssim 2^{k_1m_1+k_2m_2-jN},
\end{align*}
where we used the inequality $|\F^{-1}\psi_j(y)| \lesssim 2^{jn}(1+2^j|y|)^{-(N+n+1)}$. If we do not use the moment condition in the above argument, we have the same estimate with $j=0$. The proof of the assertion $(1)$ is complete.

Next, we will show the assertion (2). We write
\begin{align*}
 T_{\sigma_{j, \K, \Nu}}(f_1, f_2)(x) 
 = \int_{(\R^n)^2}  K_{j, \K, \Nu}(x, x-y_1, x-y_2) f_1(y_1)f_2(y_2)\, dy_1 dy_2,
\end{align*}
where $K_{j, \K, \Nu}$ is defined by
\begin{align*}
  K_{j, \K, \Nu}(x, y_1, y_2) =
  \frac{1}{(2\pi)^{2n}}\int_{(\R^n)^2} e^{i(y_1 \cdot \xi_1+y_2 \cdot \xi_2)} \sigma_{j, \K, \Nu}(x, \xi_1, \xi_2) \, d\xi_1 d\xi_2.
\end{align*}
It follows from integration by parts,
the assertion (1) and the support condition
$\supp \sigma_{j, \K, \Nu}(x,\cdot,\cdot) \subset (\nu_1, \nu_2) + [-1, 1]^{2n}$
for each $x \in \R^n$
that
\begin{align*}
 |K_{j, \K, \Nu}(x, y_1, y_2)| \lesssim 2^{k_1m_1+k_2m_2 -jN}(1+|y_1|)^{-L} (1+|y_2|)^{-L},
\end{align*}
which gives the desired result.
\end{proof}

The following lemma plays a crucial role in the proof of Theorem \ref{main1},
and its essence goes back to \cite[Lemma 3.6]{MT}.

\begin{lem} \label{lem-L2L2Lr-est}
Let $2 \leq r \leq \infty$, and let $\Lambda$ be a finite subset of $\Z^n$.
For each $N \ge 0$, we have
 \begin{multline*} 
  \Big(
  \sum_{\nu_1 \in \Lambda} \sum_{\nu_2 \in \Z^n} + 
  \sum_{\nu_1 \in \Z^n} \sum_{\nu_2 \in \Lambda} +
  \sum_{\mu \in \Lambda} \sum_{\nu_1+\nu_2 = \mu}
  \Big)
  |\la T_{\sigma_{j, \K, \Nu}}(f_1, f_2), g \ra|
  \\
 \lesssim 2^{k_1 m_1+k_2 m_2 -jN}
 |\Lambda|^{1/2} \|f_1\|_{L^2}\|f_2\|_{L^2} \|g\|_{L^r}
 \end{multline*}
for all $j \in \N_0$ and $\K=(k_1,k_2) \in (\N_0)^2$,
where $|\Lambda|$ is the number of elements of $\Lambda$.
\end{lem}

\begin{proof}
 Let $\widetilde{\varphi}$ be a Schwartz function such that 
  $\supp \widetilde{\varphi} \subset [-2, 2]^n$ and $\widetilde{\varphi}(\xi) = 1$ on $ [-1, 1]^n$. For $\nu_i \in \Z^n$, $i=1,2$, $\mu \in \Z^n$ and $j \geq 0$, we set
 \begin{align*}
  f_{i, \nu_i} = \widetilde{\varphi}(D-\nu_i)f_i,
  \quad
  g_{j, \mu} = \widetilde{\varphi}(2^{-(j+2)}(D-\mu))g.
 \end{align*}
 Then, it follows from the identity $\varphi \widetilde{\varphi} = \varphi$ that
 \begin{equation}\label{AAA}
T_{\sigma_{j, \K, \Nu}}(f_1, f_2) = T_{\sigma_{j, \K, \Nu}}( f_{1, \nu_1}, f_{2, \nu_2}) .
 \end{equation}

The Fourier transform of $T_{\sigma_{j, \K, \Nu}}(f_1, f_2)$ can be written as
\begin{align*}
 &\F[T_{\sigma_{j, \K, \Nu}}(f_1, f_2)](\zeta)\\
 &= \frac{1}{(2\pi)^{2n}}\int_{(\R^n)^2}
   \psi_{j}(\zeta- (\xi_1+\xi_2)) [\F_x\sigma](\zeta-(\xi_1+\xi_2), \xi_1, \xi_2)\\
 &\qquad\qquad\qquad \times  \psi_{k_1}(\xi_1)\psi_{k_2}(\xi_2)
   \varphi(\xi_1-\nu_1)\varphi(\xi_2-\nu_2)
   \widehat{f_1}(\xi_1) \widehat{f_2}(\xi_2)\,d\xi_1d\xi_2,
\end{align*}
where $\F_x\sigma$ denotes the partial Fourier transform
of $\sigma(x,\xi_1,\xi_2)$ with respect to  the $x$-variable.
Now, if $\xi_i \in \supp \varphi(\cdot - \nu_i)$ for
$i=1,2 $ and $\zeta - (\xi_1+\xi_2) \in \supp \psi_j $, then
$\zeta \in \nu_1 +\nu_2 + \left[-2^{j+2}, 2^{j+2} \right]^n$. This implies that
\begin{align} \label{supp-uniform}
 \supp \F[T_{\sigma_{j, \K, \Nu}}(f_1, f_2)]
 \subset \nu_1+ \nu_2 +\left[-2^{j+2}, 2^{j+2} \right]^n.
\end{align}
Hence, it follows from \eqref{AAA} that
\begin{equation}\label{trilin-rep}
\begin{split}
\la T_{\sigma_{j, \K, \Nu}}( f_1, f_2), g \ra
&=\la T_{\sigma_{j, \K, \Nu}}( f_1, f_2), g_{j,\nu_1+\nu_2} \ra
\\
&= \la T_{\sigma_{j, \K, \Nu}}( f_{1, \nu_1}, f_{2, \nu_2}), g_{j, \nu_1+\nu_2} \ra.
\end{split}
\end{equation}

First, we consider the sum $\sum_{\nu_1 \in \Lambda} \sum_{\nu_2 \in \Z^n}$.
By \eqref{trilin-rep},
Lemma \ref{lem-symb-est} (2) with $N$ replaced by $N+n/2+n/r$
and Schwarz's inequality, we have
\begin{align*}
&\sum_{\nu_1 \in \Lambda} \sum_{\nu_2 \in \Z^n} |\la T_{\sigma_{j, \K, \Nu}}( f_1, f_2), g\ra|
= \sum_{\nu_1 \in \Lambda} \sum_{\nu_2 \in \Z^n} |\la T_{\sigma_{j, \K, \Nu}}( f_{1, \nu_1}, f_{2, \nu_2}), g_{j, \nu_1+\nu_2}\ra| \\
 &\lesssim 2^{k_1m_1 + k_2m_2-j(N+n/2+n/r)}  
 \sum_{\nu_1 \in \Lambda} \sum_{\nu_2 \in \Z^n}
 \int_{\R^n} S(f_{1, \nu_1})(x) S(f_{2, \nu_2})(x) |g_{j, \nu_1+\nu_2}(x)| \, dx  \\
 &\leq 2^{k_1m_1 + k_2m_2-j(N+n/2+n/r)} \\
 &\quad \times
 \sum_{\nu_1 \in \Lambda} 
 \int_{\R^n} S(f_{1, \nu_1})(x) 
 \Big( \sum_{\nu_2 \in \Z^n} S(f_{2, \nu_2})(x)^2 \Big)^{1/2}
 \Big( \sum_{\nu_2 \in \Z^n} |g_{j, \nu_1+\nu_2}(x)|^2 \Big)^{1/2} \, dx \\
 &\leq 2^{k_1m_1 + k_2m_2-j(N+n/2+n/r)} |\Lambda|^{1/2}\\
 &\quad\times
 \int_{\R^n} 
\Big( \sum_{\nu_1 \in \Z^n} S(f_{1, \nu_1})(x)^2 \Big)^{1/2} 
\Big(  \sum_{\nu_2 \in \Z^n} S(f_{2, \nu_2})(x)^2 \Big)^{1/2} 
\Big( \sum_{\mu \in \Z^n} |g_{j, \mu}(x)|^2\Big)^{1/2} \, dx.
\end{align*}
It follows from Lemma \ref{lem-dyadic-est} that
\begin{align*}
 \Big( \sum_{\nu_1 \in \Z^n} S(f_{1, \nu_1})(x)^2 \Big)^{1/2} 
 &\lesssim \Big( \sum_{\nu_1 \in \Z^n} S(|f_{1, \nu_1}|^2)(x) \Big)^{1/2}
  =\Big( S\Big(\sum_{\nu_1 \in \Z^n} |f_{1, \nu_1}|^2 \Big)(x) \Big)^{1/2} \\
 &\lesssim S(S(|f_1|^2))(x)^{1/2} 
 \approx S(|f_1|^2)(x)^{1/2}
\end{align*}
and
\begin{align*}
\Big( \sum_{\mu \in \Z^n} |g_{j, \mu}(x)|^2\Big)^{1/2} 
&\lesssim 2^{(j+2)n/2}S_{2^{j+2}}(|g|^2)(x)^{1/2}
\\
&\lesssim 2^{jn} \Big( \int_{\R^n} \frac{1}{(1+2^{j}|x-y|)^{Lq}}\, dy \Big)^{1/q}
\Big( \int_{\R^n} |g(y)|^r \,dy \Big)^{1/r}
\\
&\approx 2^{jn(1/2+1/r)}\|g\|_{L^r},
\end{align*} 
where we used H\"older's inequality with $1/q+1/r=1/2$
in the second inequality.
Therefore, by Schwarz's inequality and Young's inequality, we obtain
\begin{align*}
&\sum_{\nu_1 \in \Lambda} \sum_{\nu_2 \in \Z^n} |\la T_{\sigma_{j, \K, \Nu}}( f_1, f_2), g \ra|\\
&\lesssim 2^{k_1m_1 + k_2m_2-jN} |\Lambda|^{1/2}
 \Big(\int_{\R^n}S(|f_1|^2)(x)^{1/2} S(|f_2|^2)(x)^{1/2} \, dx \Big)
 \|g\|_{L^r}\\
&\leq 2^{k_1m_1 + k_2m_2-jN} |\Lambda|^{1/2}
\|S(|f_1|^2)^{1/2}\|_{L^2} \|S(|f_2|^2)^{1/2}\|_{L^2}  \|g\|_{L^r}\\
&\approx 2^{k_1m_1 + k_2m_2-jN} |\Lambda|^{1/2}
\|f_1\|_{L^2} \|f_2\|_{L^2} \|g\|_{L^r}.
\end{align*}
In the same way, we can estimate the sum $\sum_{\nu_1 \in \Z^n} \sum_{\nu_2 \in \Lambda}$.

Next, we consider the sum $\sum_{\mu \in \Lambda} \sum_{\nu_1+\nu_2 = \mu}$.
By (\ref{trilin-rep}), Lemma \ref{lem-symb-est} (2) and Schwarz's inequality,
\begin{align*}
 &\sum_{\mu \in \Lambda} \sum_{\nu_1+\nu_2 = \mu} |\la T_{\sigma_{j, \K, \Nu}}( f_1, f_2), g \ra| 
 =\sum_{\mu \in \Lambda} \sum_{\nu_1+\nu_2 = \mu} |\la T_{\sigma_{j, \K, \Nu}}( f_{1, \nu_1}, f_{2, \nu_2}), g_{j, \nu_1+\nu_2}\ra|\\
 &\lesssim 2^{k_1m_1 + k_2m_2-j(N+n/2+n/r)}  
 \sum_{\mu \in \Lambda} \sum_{\nu_1 \in \Z^n}
 \int_{\R^n} S(f_{1, \nu_1})(x) S(f_{2, \mu-\nu_1})(x) |g_{j, \mu}(x)| \, dx\\
 &\leq 2^{k_1m_1 + k_2m_2-j(N+n/2+n/r)} \\
 &\quad \times 
 \sum_{\mu \in \Lambda} \int_{\R^n} 
 \Big( \sum_{\nu_1 \in \Z^n} S(f_{1, \nu_1})(x)^2 \Big)^{1/2}
 \Big( \sum_{\nu_1 \in \Z^n} S(f_{2, \mu-\nu_1})(x)^2 \Big)^{1/2}
 |g_{j, \mu}(x)| \, dx \\
 &\leq 2^{k_1m_1 + k_2m_2-j(N+n/2+n/r)} |\Lambda|^{1/2}\\
 &\times
 \int_{\R^n} 
\Big( \sum_{\nu_1 \in \Z^n} S(f_{1, \nu_1})(x)^2 \Big)^{1/2} 
\Big(  \sum_{\nu_2 \in \Z^n} S(f_{2, \nu_2})(x)^2 \Big)^{1/2} 
\Big( \sum_{\mu \in \Z^n} |g_{j, \mu}(x)|^2\Big)^{1/2} \, dx.
\end{align*}
The rest of the proof is the same as before. The proof is complete.
\end{proof}

\section{Proof of Theorem \ref{main1}}\label{section4}
Let $m_1$, $m_2$, $\sigma$ and $p$ be the same
as in Theorem \ref{main1}.
To obtain Theorem \ref{main1}, by duality, 
it is sufficient to prove that
\begin{align*}
 |\la T_\sigma(f_1, f_2), g \ra| \lesssim \|f_1\|_{L^2} \|f_2\|_{L^2} \|g\|_{B^0_{p^\prime, \infty}}.
\end{align*}
From (\ref{decomp-symb}), we can write
\begin{align*}
  &\la T_\sigma(f_1, f_2), g \ra
  = \sum_{j\ge 0} \sum_{\K \in (\N_0)^2} \sum_{\Nu \in (\Z^n)^2}
 \la T_{\sigma_{j, \K, \Nu}}(f_1, f_2), g \ra \\
  &=  \sum_{j\ge 0} 
  \sum_{k_1 \geq k_2} 
  \sum_{\Nu \in (\Z^n)^2}
  \la T_{\sigma_{j, \K, \Nu}}(f_1, f_2), g \ra
  + 
  \sum_{j\ge 0}
  \sum_{k_1 < k_2} 
  \sum_{\Nu \in (\Z^n)^2}
  \la T_{\sigma_{j, \K, \Nu}}(f_1, f_2), g \ra.
\end{align*}
By symmetry, we only consider the former sum in the last line, because the argument below works for the latter one.

Let $\widetilde{\psi}_k  \in \Sh(\R^n)$, $k \geq 0$, be such that  
\begin{align*}
&\supp \widetilde{\psi}_0 \subset \{ |\xi| \leq 4 \}, \quad
\supp \widetilde{\psi}_k \subset \{2^{k-2} \leq |\xi| \leq 2^{k+2}\},\quad k \geq 1,\\
&\widetilde{\psi}_k = 1 
\quad \text{on} \quad
\supp \psi_k, \quad k \geq 0.
\end{align*}
Since
$\psi_{k_i}\widetilde{\psi}_{k_i}=\psi_{k_i}$,
it holds that
\begin{align*}
  \la T_{\sigma_{j, \K, \Nu}}(f_1, f_2), g \ra 
  =  \la T_{\sigma_{j, \K, \Nu}}(f_{1, k_1}, f_{2, k_2}), g \ra
\end{align*}
with 
$f_{i, k_i} = \widetilde{\psi}_{k_i}(D)f_i$, $i=1,2$.
We  also use the decomposition
\begin{align*}
 g = \sum_{\ell \geq 0} \psi_\ell(D)g = \sum_{\ell \geq 0} g_\ell,
\end{align*}
where $\{\psi_\ell\}_{\ell \geq 0}$ is the same as in (\ref{Besov-parti}).
Then, we can write
\begin{align*}
\sum_{j\ge 0} 
  \sum_{k_1 \geq k_2} 
  \sum_{\Nu \in (\Z^n)^2}
  \la T_{\sigma_{j, \K, \Nu}}(f_1, f_2), g \ra
= \sum_{j\ge 0} 
  \sum_{k_1 \geq k_2} 
  \sum_{\ell \geq 0}
  \sum_{\Nu \in (\Z^n)^2}
  \la T_{\sigma_{j, \K, \Nu}}(f_{1, k_1}, f_{2, k_2}), g_\ell \ra.
 \end{align*}
Furthermore, we divide the sum as follows.
\begin{align*}
  &\sum_{j \ge 0} 
  \sum_{k_1 \geq k_2} 
  \sum_{\ell \geq 0}
  \sum_{\Nu \in (\Z^n)^2}
  \la T_{\sigma_{j, \K, \Nu}}(f_{1, k_1}, f_{2, k_2}), g_\ell \ra\\
 &= \Big( \sum_{\substack{j \geq k_1-3\\ k_1 \geq k_2}} 
     + \sum_{\substack{j < k_1-3\\ k_1 \geq k_2}} \Big)
  \sum_{\ell \geq 0}
  \sum_{\Nu \in (\Z^n)^2}
 \la T_{\sigma_{j, \K, \Nu}}(f_{1, k_1}, f_{2, k_2}), g_\ell \ra
 = A_1 + A_2.
\end{align*}

\bigskip
\noindent
{\it Estimate for $A_1$}.
Since $\supp \widehat{g_\ell} \subset \{ 2^{\ell-1} \le |\zeta| \le 2^{\ell+1}\}$, $\ell \ge 1$, and
\begin{align*}
\supp \F[ T_{\sigma_{j, \K, \Nu}}(f_{1, k_1}, f_{2, k_2})] 
 \subset \{ |\zeta| \leq 2^{j+6}\}, \quad  j\geq k_1-3, \ k_1 \geq k_2
\end{align*}
 (see the argument around (\ref{supp-uniform})),
 it follows that $\la T_{\sigma_{j, \K, \Nu}}(f_{1, k_1}, f_{2, k_2}), g_\ell \ra = 0$ if $\ell \geq j+7$.
In addition, we see that if $\supp \varphi(\cdot - \nu_2) \cap \supp 
\psi_{k_2} = \emptyset$, then $\sigma_{j, \K, \Nu} = 0$, and consequently $\la T_{\sigma_{j, \K, \Nu}}(f_{1, k_1}, f_{2, k_2}), g_\ell \ra = 0$. 
From these observations, $A_1$ can be written as
\begin{align*}
A_1
= \sum_{\substack{j \geq k_1-3 \\ \ell \leq j+6, \, k_1 \geq k_2}} 
  \sum_{\nu_1 \in \Z^n} \sum_{\nu_2 \in \Lambda_{k_2}}
 \la T_{\sigma_{j, \K, \Nu}}(f_{1, k_1}, f_{2, k_2}), g_\ell \ra,
\end{align*}
where 
\begin{align} \label{k2set}
 \Lambda_{k_2} = \{\nu_2 \in \Z^n : \supp \varphi(\cdot - \nu_2) \cap \supp \psi_{k_2} \neq \emptyset \}.
\end{align}
Note that the number of elements of $\Lambda_{k_2}$ satisfies $|\Lambda_{k_2}| \lesssim 2^{k_2n}$.
Using Lemma \ref{lem-L2L2Lr-est} with $r = p^\prime$ and $N \ge 1$, we obtain
\begin{align} \label{A1-est}
\begin{split}
 |A_1| 
 &\leq
 \sum_{\substack{j \geq k_1-3 \\ \ell \leq j+6, \, k_1 \geq k_2}} 
  \sum_{\nu_1 \in \Z^n} \sum_{\nu_2 \in \Lambda_{k_2}}
   |\la T_{\sigma_{j, \K, \Nu}}(f_{1, k_1}, f_{2, k_2}), g_\ell \ra|\\
 &\lesssim 
 \sum_{\substack{j \geq k_1-3 \\ \ell \leq j+6, \, k_1 \geq k_2}} 
  2^{k_1 m_1+k_2 m_2 -jN} |\Lambda_{k_2}|^{1/2} \|f_{1, k_1}\|_{L^2}\|f_{2, k_2}\|_{L^2} \|g_\ell\|_{L^{p^\prime}}\\
 &\lesssim
 \sum_{\substack{j \geq k_1-3 \\ \ell \leq j+6, \, k_1 \geq k_2}} 
  2^{k_1 m_1+k_2 m_2+k_2n/2 -jN}\|f_{1, k_1}\|_{L^2}\|f_{2, k_2}\|_{L^2} \|g_\ell\|_{L^{p^\prime}}.
\end{split}
\end{align}
By our assumptions $m_1, m_2 < 0$ and $m_1+m_2 = -n/2$, 
it follows that $k_1 m_1+k_2 m_2+k_2n/2 \leq k_2 m_1+k_2 m_2+k_2n/2 = 0$ for $k_1 \geq k_2$. 
Hence, since $\sum_{k_i \ge 0} |\widetilde{\psi}_{k_i}(\xi_i)|^2 \lesssim 1$, the last quantity in (\ref{A1-est}) is estimated by
\begin{align*}
 &\sum_{\substack{j \geq k_1-3 \\ \ell \leq j+6, \, k_1 \geq k_2}} 
  2^{-jN}\|f_{1, k_1}\|_{L^2}\|f_{2, k_2}\|_{L^2} \|g_\ell\|_{L^{p^\prime}}\\
 &\leq 
 \sum_{\substack{j \geq k_1-3 \\  k_1 \geq k_2}} 
 2^{-jN} (j+7) \|f_{1, k_1}\|_{L^2}\|f_{2, k_2}\|_{L^2} 
 \big(\sup_{\ell \ge 0}\|g_\ell\|_{L^{p^\prime}} \big)\\
 &\leq
 \sum_{j \geq 0} 
 2^{-jN} (j+7) (j+4) 
  \Big( \sum_{k_1 \geq 0} \|f_{1, k_1}\|^2_{L^2} \Big)^{1/2}
  \Big( \sum_{k_2 \geq 0} \|f_{2, k_2}\|^2_{L^2} \Big)^{1/2}
  \|g\|_{B^0_{p^\prime, \infty}} \\
 &\lesssim \|f_1\|_{L^2} \|f_2\|_{L^2} \|g\|_{B^0_{p^\prime, \infty}},
\end{align*}
which gives the desired result.

\bigskip
\noindent
{\it Estimate for $A_2$}.
 We divide $A_2$ as follows.
\begin{align*}
A_2
 &= \Big( \sum_{k_1-3 \leq k_2 \leq k_1} + \sum_{k_2 < k_1-3}  \Big)
  \sum_{j < k_1-3}
  \sum_{\ell \geq 0}
  \sum_{\Nu \in (\Z^n)^2}
 \la T_{\sigma_{j, \K, \Nu}}(f_{1, k_1}, f_{2, k_2}), g_\ell \ra \\
 &= A_{2,1} +A_{2,2}.
\end{align*}

We first consider the estimate for  $A_{2,1}$. Since
\begin{align*}
 \supp \F[ T_{\sigma_{j, \K, \Nu}}(f_{1, k_1}, f_{2, k_2})] 
 \subset \{ |\zeta| \leq 2^{k_1+3}\}, \quad j < k_1-3, \ k_2 \le k_1,
\end{align*}
it follows that $\la T_{\sigma_{j, \K, \Nu}}(f_{1, k_1}, f_{2, k_2}), g_\ell \ra = 0$
if $\ell \geq k_1+4$.
Furthermore, by (\ref{supp-uniform}) and
the fact $\supp \widehat{g_\ell} \subset \supp \psi_\ell$,
we see that if
$\left(\nu_1+\nu_2 + \left[-2^{j+2}, 2^{j+2} \right]^n \right) \cap \supp \psi_\ell = \emptyset$,
then $\la T_{\sigma_{j, \K, \Nu}}(f_{1, k_1}, f_{2, k_2}), g_\ell \ra = 0$. 
Combining these observations,  we see that $A_{2,1}$ can be written as
\begin{align*}
 A_{2,1} 
   =\sum_{\substack{j < k_1-3, \, \ell \le k_1+3 \\ k_1-3 \leq k_2 \leq k_1}}
    \sum_{\mu \in \Lambda_{j, \ell}}
    \sum_{\nu_1 + \nu_2 = \mu}
 \la T_{\sigma_{j, \K, \Nu}}(f_{1, k_1}, f_{2, k_2}), g_\ell \ra,
\end{align*}
where
\begin{align*}
\Lambda_{j, \ell} = \{ \mu \in \Z^n : (\mu + \left[-2^{j+2}, 2^{j+2}\right]^n ) \cap \supp \psi_\ell \neq \emptyset \}.
\end{align*}
The number of elements of $\Lambda_{j, \ell}$ can be estimated by $ |\Lambda_{j, \ell}| \lesssim 2^{(j+\ell)n}$. 
Hence, it follows from Lemma \ref{lem-L2L2Lr-est} with $r = p^\prime$ and $N > n/2$ that
\begin{align*}
 |A_{2,1}| 
 &\leq \sum_{\substack{j < k_1-3, \, \ell \le k_1+3 \\ k_1-3 \leq k_2 \leq k_1}}
    \sum_{\mu \in \Lambda_{j, \ell}}
    \sum_{\nu_1 + \nu_2 = \mu}
 |\la T_{\sigma_{j, \K, \Nu}}(f_{1, k_1}, f_{2, k_2}), g_\ell \ra|\\
 &\lesssim \sum_{\substack{j < k_1-3, \, \ell \le k_1+3 \\ k_1-3 \leq k_2 \leq k_1}}
    2^{k_1m_1+k_2m_2-jN}  |\Lambda_{j, \ell}|^{1/2} 
    \|f_{1, k_1}\|_{L^2}\|f_{2, k_2}\|_{L^2} \|g_\ell\|_{L^{p^\prime}}\\
 &\lesssim \sum_{\substack{j < k_1-3, \, \ell \le k_1+3 \\ k_1-3 \leq k_2 \leq k_1}}
    2^{k_1m_1+k_2m_2-jN} 2^{(j+\ell)n/2}  
    \|f_{1, k_1}\|_{L^2}\|f_{2, k_2}\|_{L^2} \|g\|_{B^0_{p^\prime, \infty}}\\
 &\approx \sum_{ k_1-3 \leq k_2 \leq k_1}
    2^{k_1m_1+k_2m_2} 2^{k_1n/2}  
    \|f_{1, k_1}\|_{L^2}\|f_{2, k_2}\|_{L^2} \|g\|_{B^0_{p^\prime, \infty}}\\
 &= \sum_{ k_1-3 \leq k_2 \leq k_1}
    2^{-k_1m_2+k_2m_2} 
    \|f_{1, k_1}\|_{L^2}\|f_{2, k_2}\|_{L^2} \|g\|_{B^0_{p^\prime, \infty}}.
\end{align*}
Since we can write $k_2 = k_1 + \widetilde{k}$ with $\widetilde{k} = -3, -2, -1, 0$ if $k_1-3 \leq k_2 \leq k_1$, the last sum can be written as
\begin{align*}
 &\sum_{ k_1 \ge 0} \sum_{ -3 \le \widetilde{k} \le 0}
    2^{-k_1m_2+(k_1+\widetilde{k})m_2} 
    \|f_{1, k_1}\|_{L^2}\|f_{2, k_1+\widetilde{k}}\|_{L^2}\\
 &= \sum_{ k_1 \ge 0} \sum_{ -3 \le \widetilde{k} \le 0}
  2^{\widetilde{k}m_2} \|f_{1, k_1}\|_{L^2}\|f_{2, k_1+\widetilde{k}}\|_{L^2}.
\end{align*}
Hence, by Schwarz's inequality, the right hand side of the above is estimated by
\begin{align*}
\sum_{ -3 \le \widetilde{k} \le 0} 2^{\widetilde{k}m_2} 
\Big( \sum_{k_1\ge 0} \|f_{1, k_1}\|^2_{L^2} \Big)^{1/2}
\Big( \sum_{k_1\ge 0} \|f_{2, k_1+\widetilde{k}}\|^2_{L^2} \Big)^{1/2}
\lesssim \|f_1\|_{L^2}\|f_2\|_{L^2},
\end{align*}
which implies the desired estimate.

Next, we consider the estimate for $A_{2, 2}$.
Since
\begin{align*}
  \supp \F[ T_{\sigma_{j, \K, \Nu}}(f_{1, k_1}, f_{2, k_2})] 
 \subset \{2^{k_1-2} \le |\zeta| \le 2^{k_1+2}\},
\quad  j < k_1-3, \ k_2 < k_1-3,
\end{align*}it follows that
$\la T_{\sigma_{j, \K, \Nu}}(f_{1, k_1}, f_{2, k_2}), g_\ell \ra = 0$
if $\ell \le k_1-3$ or $k_1+3 \leq \ell$. 
As before, if $\supp \varphi(\cdot - \nu_2) \cap \supp \psi_{k_2} = \emptyset$, then $\la T_{\sigma_{j, \K, \Nu}}(f_{1, k_1}, f_{2, k_2}), g_\ell \ra = 0$. 
Therefore,  we obtain
\begin{align*}
 A_{2, 2} 
 = \sum_{\substack{j < k_1-3 \\ k_2 < k_1-3, \, |\ell-k_1| \le 2}}
    \sum_{\nu_1 \in \Z^n} \sum_{\nu_2 \in  \Lambda_{k_2}}
 \la T_{\sigma_{j, \K, \Nu}}(f_{1, k_1}, f_{2, k_2}), g_\ell \ra,
\end{align*}
where $\Lambda_{k_2}$ is the same as in (\ref{k2set}).
By applying Lemma \ref{lem-L2L2Lr-est}, it follows that
\begin{align} \label{A22-est}
\begin{split}
|A_{2, 2}|
&\leq \sum_{\substack{j < k_1-3 \\ k_2 < k_1-3, \, |\ell-k_1| \le 2}}
    \sum_{\nu_1 \in \Z^n} \sum_{\nu_2 \in  \Lambda_{k_2}}
 |\la T_{\sigma_{j, \K, \Nu}}(f_{1, k_1}, f_{2, k_2}), g_\ell \ra|\\
&\lesssim \sum_{\substack{j < k_1-3 \\ k_2 < k_1-3, \, |\ell-k_1| \le 2}}
 2^{k_1m_1+k_2m_2-jN} 2^{k_2n/2} \|f_{1, k_1}\|_{L^2} \|f_{2, k_2}\|_{L^2} \|g_{\ell}\|_{L^{p^\prime}}\\
&\lesssim \sum_{k_2 < k_1-3}
 2^{m_1(k_1-k_2)} \|f_{1, k_1}\|_{L^2} \|f_{2, k_2}\|_{L^2} \|g\|_{B^0_{p^\prime, \infty}}\\
&\le \sum_{k_1, k_2 \ge 0}
 2^{m_1|k_1-k_2|} \|f_{1, k_1}\|_{L^2} \|f_{2, k_2}\|_{L^2} \|g\|_{B^0_{p^\prime, \infty}}.
 \end{split}
\end{align} 
Here, it follows from our assumption $m_1 < 0$ that
\begin{align*}
 \sup_{k_1 \ge 0} \sum_{k_2 \ge 0} 2^{m_1|k_1-k_2|} < \infty
 \quad
\text{and}
\quad
 \sup_{k_2 \ge 0} \sum_{k_1 \ge 0} 2^{m_1|k_1-k_2|} < \infty.
\end{align*}
Hence, by Lemma \ref{Schur's lemma},
the last quantity in (\ref{A22-est}) can be estimated by
\begin{align*}
 \Big( \sum_{k_1 \ge 0} \|f_{1, k_1}\|^2_{L^2} \Big)^{1/2}
  \Big( \sum_{k_2 \ge 0} \|f_{2, k_2}\|^2_{L^2} \Big)^{1/2}
  \|g\|_{B^0_{p^\prime, \infty}}
\lesssim \|f_1\|_{L^2} \|f_2\|_{L^2} \|g\|_{B^0_{p^\prime, \infty}}.
\end{align*}
This completes the proof of Theorem \ref{main1}.

\section{Proof of Theorem \ref{main2}}\label{section5}

In this section, we shall prove Theorem \ref{main2}. The ``if'' part follows from Theorem \ref{main1}, (\ref{two-classes}) and the embedding
$B^0_{p, 1} \hookrightarrow B^0_{p, q}$, 
$1\le q \le \infty$. Thus, let us consider the ``only if'' part.

We assume that the $L^2\times L^2 \to B^0_{p, q}$ boundedness of all $T_\sigma$ with $\sigma \in BS^{-n/2}_{0, 0}$ holds throughout the rest of this section.
We  also take a function $\psi_0 \in \Sh(\R^n)$ satisfying $\psi_0 = 1$ on $\{|\xi| \le 2^{1/4}\}$ and $\supp \psi_0 \subset \{|\xi| \le 2^{3/4}\}$, and set $\psi_\ell (\xi) = \psi_0(2^{-\ell}\xi) - \psi_0(2^{-\ell+1}\xi)$, $\ell \ge 1$. Then $\{\psi_\ell\}_{\ell \ge 0}$ satisfies (\ref{Besov-parti}) with the support conditions replaced by
\begin{align}\label{Besov-parti-2-1}
\begin{split}
 \supp \psi_0 \subset \{ |\xi| \leq 2^{3/4}\}, \quad
\supp \psi_\ell \subset \{2^{\ell -3/4} \leq |\xi| \leq 2^{\ell + 3/4}\}, \quad \ell \ge 1.
\end{split}
\end{align}
Combining this with the condition $\sum_{\ell \ge 0}\psi_{\ell}=1$,
we see that
\begin{equation}\label{Besov-parti-2-2}
\psi_0 = 1
\ \ \text{on} \ \
\{|\xi| \leq 2^{1/4}\},
\quad
\psi_\ell  = 1
\ \ \text{on} \ \
\{2^{\ell -1/4} \leq |\xi| \leq 2^{\ell + 1/4}\},
\ \ \ell \geq 1,
\end{equation}
and use this $\{\psi_\ell\}_{\ell \ge 0}$
as the partition of unity in the definition of Besov spaces. 

\bigskip
\noindent
{\it The necessity of the condition $1 \le p \le 2$}.
First, we shall prove that $p \ge 1$.
Let $\sigma \in \Sh((\R^n)^2)$ satisfy $\sigma(\xi_1, \xi_2) = 1$ on $|(\xi_1, \xi_2)| \leq 1$,
and obviously $\sigma \in BS^{-n/2}_{0, 0}$.
For $\varphi \in \Sh(\R^n)$ satisfying
$\supp \varphi \subset \{ |\xi| \leq 1\}$,
we set
\[
\widehat{f_{i, j}}(\xi_i) = 2^{jn/2}\varphi(2^j\xi_i),
\quad i=1,2, \ j \ge 1,
\]
and note that
 $\|f_{i, j}\|_{L^2} \approx 1$. 
 Since $|(\xi_1, \xi_2)| \leq 2^{-j+1/2} \leq 1$ for $|\xi_1|, |\xi_2| \leq 2^{-j}$ and $j \ge 1$, it follows that $\sigma(\xi_1, \xi_2) \varphi(2^j\xi_1) \varphi(2^j\xi_2) = \varphi(2^j\xi_1) \varphi(2^j\xi_2)$, and consequently 
\[
T_{\sigma}(f_{1, j}, f_{2, j})(x) 
= 2^{-jn}([\F^{-1}\varphi](2^{-j}x))^2.
\]
By a change of variables, $\|T_{\sigma}(f_{1, j}, f_{2, j})\|_{L^p} \approx 2^{jn(1/p-1)}$. 
Furthermore, we have
\begin{align*}
\begin{split}
 \supp \F[T_\sigma(f_{1, j}, f_{2, j})] 
 = \supp [\varphi(2^j \cdot) * \varphi(2^j \cdot)] 
 \subset \{ |\zeta| \leq 1\}.
\end{split}
\end{align*} 
Hence, it follows from (\ref{Besov-parti-2-1}) and (\ref{Besov-parti-2-2}) that 
$\psi_\ell(D)T_\sigma(f_{1, j}, f_{2, j})$ equals $T_\sigma(f_{1, j}, f_{2, j})$ if $\ell=0$, and $0$ otherwise.
From this, we obtain
\begin{align*}
 \|T_\sigma(f_{1, j}, f_{2, j})\|_{B^0_{p, q}} 
 = \|T_\sigma(f_{1, j}, f_{2, j})\|_{L^p} \approx 2^{jn(1/p-1)}.
\end{align*}
Therefore, by our assumption, the boundedness of $T_\sigma$ gives 
\begin{align*}
2^{jn(1/p-1)} \approx  \|T_\sigma(f_{1, j}, f_{2, j})\|_{B^0_{p, q}} \lesssim  \|f_{1, j}\|_{L^2} \|f_{2, j}\|_{L^2} \approx 1.
\end{align*} 
The arbitrariness of $j \ge 1$ implies that $1/p-1 \le 0$, namely $p\ge 1$.

Next we shall prove that $p \leq 2$. Our argument is based on the proof of \cite[Proposition 5.1]{KMT}. 
Let $\psi, \varphi, \widetilde{\psi}, \widetilde{\varphi} \in \Sh(\R^n)$ be such that
\begin{align*}
&\supp \psi \subset \{ 2^{-1/8} \leq |\xi_1| \leq 2^{1/8} \}, \quad
\supp \varphi \subset \{|\xi_2| \leq 2^{-j_0}\},\\
&\supp \widetilde{\psi} \subset \{ 2^{-1/4} \leq |\xi_1| \leq 2^{1/4} \},\quad \widetilde{\psi} = 1 \quad \text{on} \quad \supp \psi, \\
&\supp \widetilde{\varphi} \subset \{ |\xi_2| \leq 2^{-j_0+1}\}, \quad 
\widetilde{\varphi} = 1 \quad \text{on} \quad \supp \varphi,
\end{align*}
where $j_0$ is a positive integer satisfying
\begin{equation}\label{proof_p2}
2^{-1/4} \le 2^{-1/8} -2^{-j_0}
\quad \text{and} \quad
2^{1/8} + 2^{-j_0} \le 2^{1/4}.
\end{equation}
We set
\begin{align*}
&\sigma(\xi_1, \xi_2) =
\sum_{k \ge 1} 2^{-kn/2} \widetilde{\psi}(2^{-k}\xi_1) \widetilde{\varphi}(2^{-k}\xi_2),\\
 &\widehat{f_{1, j}}(\xi_1) = 2^{-jn/2}\psi(2^{-j}\xi_1), \quad  \widehat{f_{2, j}}(\xi_2) = 2^{-jn/2} \varphi(2^{-j}\xi_2), 
 \quad j\ge1.
\end{align*} 
It is not  difficult to show that $\sigma \in BS^{-n/2}_{0, 0}$
(since $1+|\xi_1|+|\xi_2| \approx 2^k$
for $\xi_1 \in \mathrm{supp}\, \widetilde{\psi}(2^{-k}\cdot)$
and $\xi_2 \in \mathrm{supp}\, \widetilde{\varphi}(2^{-k}\cdot)$) and
$\|f_{1, j}\|_{L^2}, \, \|f_{2, j}\|_{L^2} \approx 1$.
Now, since  $\widetilde{\psi}(2^{-k}\xi_1) \widetilde{\varphi}(2^{-k}\xi_2) \psi(2^{-j}\xi_1) \varphi(2^{-j}\xi_2)$ equals $\psi(2^{-j}\xi_1) \varphi(2^{-j}\xi_2)$ if $k=j$,
and $0$ otherwise, it follows that
\[
T_{\sigma}(f_{1, j}, f_{2, j})(x) 
=2^{jn/2} [\F^{-1}\psi](2^j x) [\F^{-1}\varphi](2^j x).
\]
From this, we have $ \| T_\sigma(f_{1, j}, f_{2, j})\|_{L^p}  \approx 2^{jn(1/2 -1/p)}$.
Moreover, by the support conditions
of $\psi, \varphi$ and \eqref{proof_p2},
we see that
$2^{j-1/4} \le |\xi_1+\xi_2| \le 2^{j+1/4}$
for $\xi_1 \in \supp \psi(2^{-j}\cdot)$
and $\xi_2 \in \supp \varphi(2^{-j}\cdot)$,
and consequently
\begin{align*}
\supp \F[T_\sigma(f_{1, j}, f_{2, j})] 
= \supp[\psi(2^{-j}\cdot)*\varphi(2^{-j}\cdot)] 
\subset \{2^{j-1/4} \leq |\zeta| \leq 2^{j+1/4}\}.
\end{align*}
Hence,  it follows from (\ref{Besov-parti-2-1}) and (\ref{Besov-parti-2-2}) that 
$\psi_\ell(D)T_\sigma(f_{1, j}, f_{2, j})$ equals
$T_\sigma(f_{1, j}, f_{2, j})$ if $\ell=j$, and $0$ otherwise.
This yields that
\begin{align*}
 \|T_\sigma(f_{1, j}, f_{2, j})\|_{B^0_{p, q}} 
 = \|T_\sigma(f_{1, j}, f_{2, j})\|_{L^p} \approx  2^{jn(1/2 -1/p)}.
\end{align*}
Therefore, by our assumption, the boundedness of $T_\sigma$ gives
\begin{align*}
 2^{jn(1/2 -1/p)} \approx \|T_\sigma(f_{1, j}, f_{2, j})\|_{B^0_{p, q}} \lesssim  \|f_{1, j}\|_{L^2} \|f_{2, j}\|_{L^2} \approx 1.
\end{align*}
The arbitrariness of $j \ge 1$ implies that $1/2-1/p \le 0$, namely $p\le 2$.

\bigskip
\noindent
{\it The necessity of the condition $1 \le q \le \infty$}.
Finally, we shall prove that $1\leq q \leq \infty$. 
The basic idea of this part goes back to
the proofs of \cite[Lemma 6.3]{MT-IUMJ} and \cite[Theorem 1.4]{Park}.
Since there is nothing to prove if $q= \infty$, we may assume that $0<q<\infty$.
By the closed graph theorem, 
our assumption implies that
there exists a positive integer $N$ such that
\begin{equation}\label{inequality_operator_norm}
\|T_{\sigma}\|_{L^2 \times L^2 \to B_{p,q}^0}
\lesssim \max_{|\alpha|, |\beta_1|, |\beta_2| \le N}
\|
(1+|\xi_1|+|\xi_2|)^{n/2}
\partial_x^{\alpha}\partial_{\xi_1}^{\beta_1}\partial_{\xi_2}^{\beta_2}
\sigma(x,\xi_1,\xi_2)\|_{L^{\infty}_{x,\xi_1,\xi_2}}
\end{equation}
for all $\sigma \in BS^{-n/2}_{0,0}$,
where $\|T_{\sigma}\|_{L^2 \times L^2 \to B_{p,q}^0}$
denotes the operator norm of $T_{\sigma}$
(see \cite[Lemma 2.6]{BBMNT}).

We use functions
$\varphi, \widetilde{\varphi} \in \Sh(\R^n)$ such that
\begin{align*}
&\supp \varphi \subset \left[-1/4, 1/4\right]^n, \quad |\F^{-1}\varphi| \geq 1
\quad \text{on} \quad [-1, 1]^n,
\\
&\supp \widetilde{\varphi} \subset \left[-1/2, 1/2 \right]^n,
\quad \widetilde{\varphi} = 1 \quad \text{on} \quad \supp \varphi.
\end{align*}
Let $\{r_{\nu}(\omega)\}_{\nu \in \Z^n}$,
$ \omega \in [0,1]^n$,
be a sequence of Rademacher functions
(see \cite[Appendix C]{Grafakos-Classical})
enumerated in such a way that
their index set is $\Z^n$.
For $\omega \in [0,1]^n$ and $\epsilon>0$,
we set
\begin{align*}
&\sigma_{\omega}(\xi_1,\xi_2)
=\sum_{k \ge k_0}\sum_{\nu \in \Lambda_k}\sum_{\nu_1+\nu_2=\nu}
r_{\nu}(\omega)(1+|\nu_1|+|\nu_2|)^{-n/2}
\widetilde{\varphi}(\xi_1-\nu_1)\widetilde{\varphi}(\xi_2-\nu_2),
\\
&\widehat{f_i}(\xi_i)
=\sum_{\mu_i \in \Z^n}\langle \mu_i \rangle^{-n/2}
(\log \langle \mu_i \rangle)^{-(1+\epsilon)/2}\varphi(\xi_i-\mu_i),
\quad i=1,2,
\end{align*}
where $\langle \mu_i \rangle=e+|\mu_i|$,
$k_0$ is a positive integer satisfying
\begin{equation}\label{proof_k0}
2^{k-1/4} \le 2^{k-1/8}-\sqrt{n}/2,
\quad 2^{k+1/8}+\sqrt{n}/2 \le 2^{k+1/4},
\quad k \ge k_0 ,
\end{equation}
and
\[
\Lambda_k=\{\nu \in \Z^n \,:\, 2^{k-1/8} \le |\nu| \le 2^{k+1/8}\}.
\]
The number of elements of $\Lambda_k$ can be estimated by $|\Lambda_k| \approx 2^{kn}$.
It is easy to check that
\begin{equation}\label{inequality_symbol}
\max_{|\alpha|, |\beta_1|, |\beta_2| \le N}
\|
(1+|\xi_1|+|\xi_2|)^{n/2}
\partial_x^{\alpha}\partial_{\xi_1}^{\beta_1}\partial_{\xi_2}^{\beta_2}
\sigma_{\omega}(x,\xi_1,\xi_2)\|_{L^{\infty}_{x,\xi_1,\xi_2}}
\lesssim 1,
\end{equation}
but it should be emphasized that
the implicit constant in this inequality is independent of $\omega \in [0,1]^n$.
Using the facts that
$\{\langle \mu_i \rangle^{-n/2}
(\log \langle \mu_i \rangle)^{-(1+\epsilon)/2}\}$
belongs to $\ell^2(\Z^n)$
and the supports of $\varphi(\cdot-\mu_i)$, $\mu_i \in \Z^n$,
are mutually disjoint,
we also see that
$\|f_i\|_{L^2} \lesssim 1$, $i=1,2$.
Thus,
by \eqref{inequality_operator_norm} and \eqref{inequality_symbol}, 
\begin{align} \label{norm-estimate}
\|T_{\sigma_\omega}(f_1, f_2)\|_{B^0_{p, q}}
\lesssim  \|f_1\|_{L^2}\|f_2\|_{L^2}
\lesssim 1,
\quad \omega \in [0,1]^n.
\end{align}

Now, since  $ \varphi(\xi_i- \nu_i)\widetilde{\varphi}(\xi_i - \mu_i)$ is equal to  $\varphi(\xi_i- \nu_i)$ if $\nu_i=\mu_i$, and $0$ otherwise, $i= 1, 2$, 
it follows that
\begin{align*}
T_{\sigma_\omega}(f_1, f_2)(x) 
&= \sum_{k \ge k_0} \sum_{\nu \in \Lambda_k} \sum_{\nu_1+\nu_2 =\nu}
r_{\nu}(\omega)  (1+|\nu_1|+|\nu_2|)^{-n/2} 
e^{i(\nu_1+\nu_2) \cdot x} (\F^{-1}\varphi(x))^2
\\
&\qquad \times
\la \nu_1 \ra^{-n/2} ( \log \la \nu_1\ra )^{-(1+\epsilon)/2}
\la \nu_2 \ra^{-n/2} ( \log \la \nu_2\ra )^{-(1+\epsilon)/2}
\\
&=\sum_{k \ge k_0}\sum_{\nu \in \Lambda_k}
r_{\nu}(\omega) d_{\nu} e^{i \nu \cdot x} \Phi(x)^2
\end{align*}
with $\Phi=\F^{-1}\varphi$ and
\begin{align*}
d_{\nu}
&=\sum_{\mu \in \Z^n}
 (1+|\nu-\mu|+|\mu|)^{-n/2}
\\
&\qquad \qquad \times
\la \nu-\mu \ra^{-n/2} ( \log \la \nu-\mu \ra )^{-(1+\epsilon)/2}
\la \mu \ra^{-n/2} ( \log \la \mu \ra )^{-(1+\epsilon)/2} .
\end{align*}
We note that
$\mathrm{supp}\, \varphi*\varphi \subset [-1/2,1/2]^n$ and
\[
\F[T_{\sigma_\omega}(f_1, f_2)](\zeta)
=(2\pi)^{-n}
\sum_{k \ge k_0}\sum_{\nu \in \Lambda_k}
r_{\nu}(\omega) d_{\nu} [\varphi*\varphi](\zeta-\nu) .
\]
By \eqref{proof_k0},
if $\zeta - \nu \in [-1/2, 1/2]^n$, $\nu \in \Lambda_k$ and $k \ge k_0$,
then $2^{k-1/4} \leq |\zeta| \leq 2^{k+1/4}$.
Hence, for $k \ge k_0$ and $\nu \in \Lambda_k$,
it follows from (\ref{Besov-parti-2-1})
and (\ref{Besov-parti-2-2}) that
$\psi_\ell(\zeta) [\varphi* \varphi] (\zeta-\nu)$
is equal to $[\varphi* \varphi] (\zeta-\nu)$ 
if $k=\ell$, and $0$ otherwise.
This yields that
\[
\psi_\ell(D)T_{\sigma_\omega}(f_1, f_2)(x) 
= \sum_{\nu \in \Lambda_\ell}
r_{\nu}(\omega) d_{\nu} e^{i \nu \cdot x} \Phi(x)^2,
\quad \ell \ge k_0.
\]

Raising \eqref{norm-estimate} to the $q$-th power
and integrating over $\omega \in [0,1]^n$,
we have
\begin{align*}
\begin{split}
1 
&\gtrsim
\int_{[0,1]^n} \|T_{\sigma_\omega}(f_1, f_2)\|_{B^0_{p, q}}^q\, d\omega
=\sum_{\ell \ge 0} \int_{[0, 1]^n}  \|\psi_\ell(D)T_{\sigma_\omega}(f_1, f_2)
\|_{L^p}^q \,d\omega
\\
&\ge
\sum_{\ell \ge k_0}
\left( \int_{\R^n}
\bigg( \int_{[0, 1]^n}  
\Big| \sum_{\nu \in \Lambda_\ell}  r_{\nu}(\omega)
d_\nu e^{i\nu \cdot x} \Phi(x)^2 \Big|^{\min\{p, q\}}
d\omega \bigg)^{p/\min\{p, q\}}
 dx \right)^{q/p}
\\
&\ge
\sum_{\ell \ge k_0}
\left( \int_{[-1,1]^n}
\bigg( \int_{[0, 1]^n}  
\Big| \sum_{\nu \in \Lambda_\ell}  r_{\nu}(\omega)
d_\nu e^{i\nu \cdot x} \Big|^{\min\{p, q\}}
d\omega \bigg)^{p/\min\{p, q\}}
dx \right)^{q/p} ,
\end{split}
\end{align*}
where
we used H\"older's inequality when $p \leq q$,
or Minkowski's inequality when $p\geq q$
in the second inequality,
and the condition $|\Phi|=|\F^{-1}\varphi| \ge 1$
on $[-1,1]^n$ in the third inequality.
Furthermore,
by Khintchine's inequality
(see, e.g., \cite[Appendix C]{Grafakos-Classical}),
the last quantity is equivalent to
\begin{align*} 
\sum_{\ell \ge k_0}
\left( \int_{[-1, 1]^n}
\Big( \sum_{\nu \in \Lambda_{\ell}} 
|d_\nu e^{i\nu \cdot x}|^2 \Big)^{p/2} \,dx \right)^{q/p} 
=2^{nq/p}\sum_{\ell \ge k_0} \Big( \sum_{\nu \in \Lambda_\ell}   |d_\nu|^2 \Big)^{q/2} .
\end{align*}
Since $|\nu-\mu| \approx |\nu|$
for $|\mu| \le |\nu|/2$, we have
\begin{align*}
|d_\nu| 
&\gtrsim \sum_{ |\mu| \le |\nu|/2 }  \la \nu \ra^{-3n/2}  
(\log\la \nu \ra)^{-(1+\epsilon)}
\approx \la \nu \ra^{-n/2}  (\log\la \nu \ra)^{-(1+\epsilon)}
\approx 2^{-\ell n/2}\ell^{-(1+\epsilon)}
\end{align*}
for $\nu \in \Lambda_\ell$,
and consequently
\begin{align*} 
 \Big( \sum_{\nu \in \Lambda_{\ell}}   |d_\nu |^2 \Big)^{q/2} 
 \gtrsim \Big( \sum_{\nu \in \Lambda_{\ell}} 2^{-\ell n} \ell^{-2(1+\epsilon)} \Big)^{q/2}
 \approx \ell^{-q(1+\epsilon)}.
\end{align*}
Combining the above estimates, we obtain
\[
\sum_{\ell \ge k_0} \ell^{-q(1+\epsilon)} \lesssim 1,
\]
which is possible only when $q(1+ \epsilon) > 1$. 
Therefore, the arbitrariness of $\epsilon > 0$ implies that $q \ge 1$.
The proof of the ``only if'' part of Theorem \ref{main2} is complete.

\section{The amalgam space $(L^2,\ell^1)$ and Wiener amalgam space $W_{1,2}$}
\label{section6}
In this section,
we shall prove Proposition \ref{main3}.
Let $1 \le p \le 2$.
It is easy to see that the embedding $B_{p,1}^0 \hookrightarrow (L^2,\ell^1)$
does not hold.
In fact, it follows from the embedding
$(L^2,\ell^1) \hookrightarrow L^q$, $1 \le q \le 2$,
(see Section \ref{section2}) that
if $B_{p,1}^0 \hookrightarrow (L^2,\ell^1)$,
then $B_{p,1}^0 \hookrightarrow L^q$ for all $1 \le q \le 2$.
However, since  $B_{p,1}^0 \hookrightarrow L^q$
if and only if $p=q$ (see, e.g., \cite[Corollary 2.2.4/2]{RS}),
this is  a contradiction.
Though we can directly prove that
the embedding $(L^2,\ell^1) \hookrightarrow B_{p,1}^0$
does not hold,
we shall show it by using the characterization of $(L^2,\ell^1)$,
Proposition \ref{amalgam-Wiener} below.
To do this, we recall the definitions
of modulation spaces and Wiener amalgam spaces
based on Feichtinger \cite{Feichtinger} and Gr\"ochenig \cite{Grochenig}.

Fix a function $\phi \in \Sh(\R^n)\setminus \{ 0 \}$
(called the window function).
Then the short-time Fourier transform $V_{\phi}f$ of
$f \in \Sh'(\R^n)$ with respect to $\phi$
is defined by
\[
V_{\phi}f(x,\xi)
=\int_{\R^n}f(t) \overline{\phi(t-x)} e^{-it\cdot \xi}\, dt  ,
\quad  x, \xi \in \R^n,
\]
where the integral is understood in the distributional sense.
For $1\le p,q \le \infty$,
the modulation space $M_{p,q}(\R^n)$
consists of all $f \in \Sh'(\R^n)$
such that
\[
\|f\|_{M_{p,q}}
=\left\{ \int_{\R^n} \left(
\int_{\R^n} |V_{\phi}f(x,\xi)|^{p}\, dx
\right)^{q/p} d\xi \right\}^{1/q}
< \infty
\]
with usual modification when $p=\infty$ or $q=\infty$.
On the other hand,
the Wiener amalgam space $W_{p,q}(\R^n)$ consists of
all $f \in \Sh'(\R^n)$ such that
\[
\|f\|_{W_{p,q}}
=\left\{ \int_{\R^n} \left(
\int_{\R^n} |V_{\phi}f(x,\xi)|^{q}\, d\xi
\right)^{p/q} dx \right\}^{1/p}
<\infty .
\]
These definitions are independent
of the choice of the window function
$\phi \in \Sh(\R^n)\setminus \{ 0 \}$, that is, different window functions yield equivalent norms.
Since $|V_{\phi}f(x,\xi)|=(2\pi)^{n}|V_{\F^{-1}\phi}[\F^{-1}f](-\xi,x)|$,
we see that
\begin{equation}\label{MW-Fourier}
\|f\|_{W_{p,q}} \approx \|\F^{-1}f\|_{M_{q,p}}.
\end{equation}
It is also known that
if $\varphi \in \Sh(\R^n)$
has a compact support and
if there exists a constant $C>0$ such that
$\left|\sum_{\nu \in \Z^n}\varphi(\xi-\nu)\right| \ge C$
for all $\xi \in \R^n$,
then
\begin{equation}\label{MW-equi}
\begin{split}
&\|f\|_{M_{p,q}} \approx
\bigg( \sum_{\nu \in \Z^n}\|\varphi(D-\nu)f\|_{L^p}^q\bigg)^{1/q},
\\
&\|f\|_{W_{p,q}} \approx
\bigg\|\bigg(\sum_{\nu \in \Z^n}|\varphi(D-\nu)f|^q\bigg)^{1/q}\bigg\|_{L^p}
\end{split}
\end{equation}
(see, e.g., \cite[Theorem 3]{Triebel_paper}).

Although the following may be well known to many people,
we shall give a proof for the reader's convenience.

\begin{prop}\label{amalgam-Wiener}
The amalgam space $(L^2,\ell^1)(\R^n)$
coincides with the Wiener amalgam space $W_{1,2}(\R^n)$.
\end{prop}

\begin{proof}
Let $\varphi$ be a nonnegative function in $\Sh(\R^n)$
such that $\mathrm{supp}\, \varphi$ is compact and
$\varphi \ge 1$ on $[-1/2,1/2]^n$.
Then, we have the norm equivalence
\[
\|f\|_{(L^2,\ell^1)}
\approx
\sum_{\nu \in \Z^n}
\|\varphi(\cdot-\nu)f \|_{L^2}
\]
(see, e.g., \cite[Lemma 2.1]{KMT-2}).
Since $\sum_{\nu \in \Z^n}\varphi(\xi-\nu) \ge 1$
for all $\xi \in \R^n$,
by Plancherel's theorem,
\eqref{MW-equi} and \eqref{MW-Fourier},
\begin{align*}
\sum_{\nu \in \Z^n}
\|\varphi(\cdot-\nu)f \|_{L^2}
&\approx \sum_{\nu \in \Z^n}
\|\F^{-1}[\varphi(\cdot-\nu)f ]\|_{L^2}
\\
&= \sum_{\nu \in \Z^n}
\|\varphi(D-\nu)[\F^{-1}f]\|_{L^2}
\approx \|\F^{-1}f\|_{M_{2,1}}
\approx \|f\|_{W_{1,2}}.
\end{align*}
The proof is complete.
\end{proof}

In the rest of this section,
we shall prove that
the embedding $(L^2,\ell^1) \hookrightarrow B_{p,1}^0$
does not hold, where $1 \le p \le 2$.
Let $\phi, \varphi$ be nonnegative functions in $\Sh(\R^n)$ such that
$\mathrm{supp}\, \phi \subset [-1/4,1/4]^n$,
$\mathrm{supp}\, \varphi \subset [-3/4,3/4]^n$
and $\varphi=1$ on $[-1/2,1/2]^n$.
We take a positive integer $N_0$ satisfying
\begin{equation}\label{prop1.3-A}
2^{k-1/4} \le 2^{k}-\sqrt{n}/4,\quad 2^{k}+\sqrt{n}/4 \le 2^{k+1/4},
\quad k \ge N_0,
\end{equation}
and set
\[
f_N(x)=\sum_{k=N_0}^{N}e^{i2^k e_1\cdot x}\F^{-1}\phi(x),
\quad N \ge N_0,
\]
where $e_1=(1,0,\dots,0) \in \R^n$.

We first consider the Besov space norm of $f_N$.
The support condition of $\phi$ and \eqref{prop1.3-A} give
\[
\mathrm{supp}\, \phi(\cdot-2^k e_1)
\subset \{|\xi-2^{k}e_1| \le \sqrt{n}/4\}
\subset \{2^{k-1/4} \le |\xi| \le 2^{k+1/4}\},
\quad k \ge N_0.
\]
Then,
since $\widehat{f_N}(\xi)=\sum_{k=N_0}^{N}\phi(\xi-2^{k}e_1)$,
we see that
$\psi_{\ell}(D)f_N(x)$ is equal to
$e^{i2^\ell e_1\cdot x}\F^{-1}\phi(x)$ if $N_0 \le \ell \le N$,
and $0$ otherwise,
where $\{\psi_{\ell}\}$ is the same as in the beginning of Section \ref{section5},
and we used (\ref{Besov-parti-2-1}) and (\ref{Besov-parti-2-2}).
Hence, 
\begin{equation}\label{estimate-Besov}
\|f_N\|_{B_{p,1}^0}
=\sum_{\ell=N_0}^{N}
\|e^{i2^\ell e_1\cdot x}\F^{-1}\phi\|_{L^p}
\approx N
\end{equation}
for sufficiently large $N$.

We next consider the amalgam space norm of $f_N$.
By Proposition \ref{amalgam-Wiener},
it is sufficient to estimate $\|f_N\|_{W_{1,2}}$.
Moreover, since $\mathrm{supp}\, \varphi$ is compact
and $\sum_{\nu \in \Z^n}\varphi(\xi-\nu) \ge 1$,
we have by \eqref{MW-equi}
\[
\|f_N\|_{(L^2,\ell^1)}
\approx \bigg\|\bigg(\sum_{\nu \in \Z^n}
|\varphi(D-\nu)f_N|^2\bigg)^{1/2}\bigg\|_{L^1}.
\]
Noting that
$\mathrm{supp}\, \phi(\cdot-2^k e_1) \subset 2^k e_1+[-1/4,1/4]^n$,
$\mathrm{supp}\, \varphi(\cdot-\nu) \subset \nu+[-3/4,3/4]^n$
and $\varphi(\cdot-\nu)=1$ on $\nu+[-1/4,1/4]^n$,
we see that $\varphi(D-\nu)f_N(x)$ is equal to
$e^{i2^k e_1\cdot x}\F^{-1}\phi(x)$ if $\nu=2^{k}e_1$ and $N_0 \le k \le N$,
and $0$ otherwise.
Therefore,
\begin{equation}\label{estimate-amalgam}
\|f_N\|_{(L^2,\ell^1)}
\approx \bigg\|\bigg(\sum_{k=N_0}^{N}
|e^{i2^k e_1\cdot x}\F^{-1}\phi|^2\bigg)^{1/2}\bigg\|_{L^1}
\approx N^{1/2}
\end{equation}
for sufficiently large $N$.

It follows from \eqref{estimate-Besov} and \eqref{estimate-amalgam} that
if $(L^2,\ell^1) \hookrightarrow B_{p,1}^0$,
then
\[
N \approx \|f_N\|_{B_{p,1}^0} \lesssim \|f_N\|_{(L^2,\ell^1)} \approx N^{1/2}
\]
for all sufficiently large $N$. However, this is a contradiction.
The proof of Proposition \ref{main3} is complete.

\section*{Acknowledgement}
The authors thank the referee for his/her careful reading of the manuscript. 
The research of the third author was partially supported
by JSPS KAKENHI Grant Number JP20K03700.


\end{document}